\documentclass[12pt]{amsart}
\usepackage{amssymb,amsmath,amsthm,fullpage}
\usepackage{enumerate,color,pdfsync}

\numberwithin{equation}{section}

\newcommand{\abs}[1]{\left\vert#1\right\vert}
\newcommand{\set}[1]{\left\{#1\right\}}

\newcommand{\eps}{\varepsilon}

\newcommand{\re}{\mathrm{Re}}

\renewcommand{\L}{\mathcal{L}}

\newcommand{\C}{\mathbb C}

\DeclareMathOperator{\Ran}{Range}

\newcommand{\bd}{\textrm{b}}

\newcommand{\p}{\partial}

\newcommand{\z}{\bar z}

\newcommand{\dbar}{\bar\partial}

\newcommand{\dbarb}{\bar\partial_b}

\DeclareMathOperator{\Tr}{Tr}

\newcommand{\la}{\langle}
\newcommand{\ra}{\rangle}

\newcommand{\LL}{\bar L}
\DeclareMathOperator{\dist}{dist}
\DeclareMathOperator{\Span}{span}
\DeclareMathOperator{\Hess}{Hess}
\DeclareMathOperator{\rank}{rank}
\DeclareMathOperator{\opL}{\mathcal{L}}

\newtheorem{thm}{Theorem}[section]
\newtheorem{prop}[thm]{Proposition}
\newtheorem{lem}[thm]{Lemma}
\newtheorem{cor}[thm]{Corollary}

\newtheorem*{theorem*}{Theorem}

\theoremstyle{definition}
\newtheorem{defn}[thm]{Definition}
\newtheorem{ex}[thm]{Example}

\theoremstyle{remark}
\newtheorem{rem}[thm]{Remark}

\newcommand{\Om}{\Omega}

\begin{document}

\title{Boundary invariants and the closed range property for $\dbar$}%

\author{Phillip S. Harrington and Andrew Raich}%
\address{SCEN 309, 1 University of Arkansas, Fayetteville, AR 72701}%
\email{psharrin@uark.edu, araich@uark.edu}%

\thanks{The second author is partially supported by NSF grant DMS-1405100.}%

\begin{abstract}
This paper provides a connection between two distinct branches of research in CR geometry -- namely, analytic and geometric conditions that suffice
to establish the closed range of the
Cauchy-Riemann operator and CR invariants on CR manifolds. Specifically, we work on not necessarily pseudoconvex domains $\Om\subset\C^n$ and
define third and fourth order CR invariants on $\bd\Om$ and show that these invariants provide enough information to establish 
closed range for the $\dbar$-Laplacian in $L^2_{(0,q)}(\Om)$ for a given, fixed $q$. The closed range estimates 
follow from our previously defined weak $Z(q)$ condition. We also develop powerful linear
algebra machinery to translate
the information from the invariants into information about the Levi form and its eigenvalues.
We conclude with several examples that demonstrate the usefulness and ease of use of the new conditions.
\end{abstract}

\maketitle

%
%
\section{Introduction}
The main goal of this paper is to demonstrate that geometric information captured by certain invariant CR tensors
provides sufficient information to establish the closed range property for
$\dbar$ on a domain $\Omega\subset\C^n$.  A secondary goal of the paper is to provide
general construction methods for establishing when a domain (or its boundary) satisfies weak $Z(q)$, which is a sufficient condition for closed range introduced by the authors in \cite{HaRa15}.

The theory for the closed range property for $\dbar$ on pseudoconvex domains
is classical and well-understood \cite{Hor65} (although see \cite{HeMc16} for additional discussion of the unbounded case), so we focus on non-pseudoconvex domains. The analog of strict pseudoconvexity at the level
of $(0,q)$-forms is $Z(q)$, a geometric condition that says (for domains $\Om\subset\C^n$) that at every point the Levi form on $\bd\Om$ has at least
$(n-q)$ positive or at least $(q+1)$ negative eigenvalues \cite{Hor65}. H\"ormander also proved that
$Z(q)$ is equivalent to the $\dbar$-Neumann operator gaining $1/2$ a derivative in $L^2_{0,q}(\Om)$ Sobolev spaces. While there is no known necessary
and sufficient condition for closed range of $\dbar$ on $(0,q)$-forms for fixed $q$, the most general
sufficient condition for closed range of $\dbar$ on $(0,q)$-forms is the weak $Z(q)$ condition that we introduced in \cite{HaRa15}.

Our approach in this paper is novel relative to earlier efforts to establish closed range for $\dbar$ on $(0,q)$-forms
\cite{HaRa11,HaRa15, Ho91,Sha85a}. In the earlier papers, the authors manipulated the basic identity/basic estimate to find a version that was
well-adapted to the geometry of their particular hypotheses. Here, we take an alternative approach -- we start with geometric information about
$\bd\Om$ encoded by the Levi form and its derivatives (i.e., the tensors mentioned above) and seek to convert it to information about weak $Z(q)$.

To find an appropriate hypothesis for $\Om$, we start by further exploring the $Z(q)$ condition.
That the boundary of every bounded domain contains
a point of convexity means that for bounded $Z(q)$ domains with connected boundaries the Levi form must have at least $n-q$ positive eigenvalues at each boundary point. Generically, if
the Levi form for a bounded domain has a degeneracy from $Z(q)$, it would be that the Levi form has least $n-q-1$ positive eigenvalues,
so this is our starting positivity hypothesis. Additionally, the Levi form of a bounded domain that satisfies $Z(q)$ can have at most $q-1$ negative
eigenvalues at each boundary point, so our negativity hypothesis is that the Levi form has at most $q-1$ negative eigenvalues.


This work was inspired by the recent preprint of Zaitsev \cite{Zai19}, who seeks to build a family of invariants that will make Catlin's seminal works
on finite type and subelliptic estimates for $\dbar $ on pseudoconvex domains in $\C^n$ more accessible to the general mathematical community.
Zaitsev therefore concentrates on pseudoconvex domains and consequently his theory is not applicable to the non-pseudoconvex case. Roughly speaking,
our CR invariants will measure the degree to which the Levi form vanishes in particular directions at points in the null space of the Levi form $\opL$. We define
two invariant tensors $\tau^3$ and $\tau^4$. Essentially, $\tau^3$ measures first order vanishing in directions in the kernel of the Levi-form. The Levi form is
already a $(1,1)$-form whose coefficients are comprised of second order derivatives of the defining function of $\Omega$.
The fourth order invariant $\tau^4$ is defined on the
null space of $\tau^3$ and generated by differentiating the Levi form a second time. Generally speaking, one can build a sequence of invariants
so that the $(k+1)$st order invariant is defined by differentiating the Levi form (or the $k$th order invariant) by directions in the null space of the $k$th order invariant
(see, e.g., \cite{Ebe98}). Thus, the boundary invariants $\opL, \tau^3$, and $\tau^4$ give information about the curvature of $\bd\Om$.

It is well-known that solvability of the $\dbar$-equation with estimates is deeply intertwined with curvature conditions of $\bd\Om$. Pseudoconvexity is
simply nonnegativity of the Levi form. When pseudoconvexity does not hold, information about the Levi form and directions in which its
eigenvalues are positive, negative, or
vanishing is critical in establishing a good basic estimate. The $(1,1)$-vector $\Upsilon$ that we define in \cite{HaRa15} in Stein manifolds
and more simply for domains in $\C^n$ in \cite{HaRa19,HaRa18} exactly captures this type of information and was specifically designed to provide
an avenue for proving closed range estimates. The problem with $\Upsilon$,
however, is that there are not general techniques to build it, except when the hypotheses
are very strong, like $Z(q)$ or uniform $Z(q)$ \cite{HaRa11,HaRa17z}.

Despite the fact that invariants and closed range estimates both require information on the Levi form and the curvature of $\bd\Om$, the development of these two
theories happened in parallel, with little cross pollination when $n \geq 3$. On pseudoconvex domains in $\C^2$,
the situation is different as the control metric on finite type domains is well known to govern the behavior of all operators related to the $\dbar$ and
$\dbarb$. The key elements in computing the control metric are commutators of tangential vectors fields, and these turn out to give exactly the
derivatives of the Levi form. See, for example, \cite{NaStWa85,NaRoStWa89,McN89}.

Our considerations are substantively different than the $\C^2$ case because the interest there was in pointwise and subelliptic estimates, and the degree
to which the Levi form vanished at weakly pseudoconvex points directly informed the estimate. We are interested in proving closed range, and
the lack of pseudoconvexity plays a fundamental role in terms of the estimates that one can prove. In \cite{Zai19}, pseudoconvexity is a critical hypothesis,
so our development of a fourth order invariant $\tau^4$ is independent of his, and indeed, is neither more nor less general.
Although we use the same notation, this is due to the similarity
of the underlying motivation and not the actual details of the construction. The two invariants are quite different.

Our original result in \cite{HaRa15} requires the existence of a matrix of $C^2$ functions $\Upsilon$ satisfying a family of properties.
The main result in these notes, Theorem \ref{thm:main_theorem}, could be expressed by saying that we need a (constant) matrix $N$ satisfying a family of properties.
This is already more ``checkable," since we are searching in a finite dimensional space of objects at the expense of a little bit of generality.
In order to further understand this constant matrix $N$, we develop some linear algebra that allows us to formulate several theorems that
connect the invariant tensors with weak $Z(q)$.
Theorem \ref{thm:isolated degeneracy} is the most ``checkable" of the conditions, as demonstrated by Example \ref{ex:simple example}.
Theorem \ref{thm:isolated degeneracy_improved} invokes additional machinery to derive a more difficult condition to check,
but Example \ref{ex:complicated example} shows that the generality admits further examples.  We are not aware of another paper that merges CR geometry with sufficient conditions for closed range in this way.

The outline of the paper is as follows. Section \ref{sec:results} contains the main results and the definitions required to properly formulate those results.
In Section \ref{sec:main_proofs}, we prove the main result.
There is a lot of geometry here, mostly to make sure that $\tau^4_p$ is correctly defined.  Section \ref{sec:indefinite_subspaces} sets up the multilinear algebraic machinery (including an exhaustive list of the examples with $2\times 2$ matrices).  With this new material to characterize the possible range of $N$, we are able to prove Theorems \ref{thm:isolated degeneracy} and \ref{thm:isolated degeneracy_improved}.

We conclude the paper with Section \ref{sec:example} in which we provide two examples -- Example \ref{ex:simple example} which demonstrates the
usefulness of Theorem \ref{thm:isolated degeneracy} and Example \ref{ex:complicated example} which shows the power of the more general formulation.
In both cases, we explicitly compute $\Upsilon$ using the methods of proof of the paper, so reading these examples can shine a light on the thought processes
behind the analysis.

%
%
\section{The Main Results}\label{sec:results}

\subsection{Definitions}\label{subsec:definitions}

\subsubsection{CR Geometry}
For a function $\alpha$, we denote $\alpha_k = \frac{\p\alpha}{\p z_k}$ and $\alpha_{\bar j} = \frac{\p\alpha}{\p\z_j}$.

Let $\Omega\subset\mathbb{C}^n$ be a bounded domain with $C^4$ boundary, and choose a $C^4$ defining function $\rho$ for $\Omega$.
We normalize our metric so that $|dz_j|=1$ for all $1\leq j\leq n$.  Let $T(\bd\Om)$ denote the space of tangent vectors on $\bd\Om$,
with $\mathbb{C}T(\bd\Om)$ denoting its complexification.  For $X\in T(\mathbb{C}^n)$, we denote the Levi-Civita connection by $\nabla_X$.
When needed, we will extend this to $\mathbb{C}T(\mathbb{C}^n)$ by requiring it to be $\mathbb{C}$-linear.
Since we are using the Euclidean metric on $\mathbb{C}^n$, in orthonormal coordinates we have
\[
  \nabla_X\left(\sum_{j=1}^n\left( a^j\frac{\partial}{\partial z_j}+b^j\frac{\partial}{\partial\bar z_j}\right)\right)
  =\sum_{j=1}^n\left( (X a^j)\frac{\partial}{\partial z_j}+(X b^j)\frac{\partial}{\partial\bar z_j}\right).
\]
For $X\in T(\bd\Om)$, we will use $\nabla^b_X$ to denote the Levi-Civita connection induced by the restriction to $\bd\Om$.  Let $\nu=2|d\rho|^{-1}\re\left(\sum_{j=1}^n\frac{\partial\rho}{\partial\bar z_j}\frac{\partial}{\partial z_j}\right)$ denote the unit outward normal to $\bd\Om$.
For $X,Y\in T(\bd\Om)$, we may check that $\nabla^b_X Y$ is the tangential part of $\nabla_X Y$. In other words,
\[
  \nabla^b_X Y=\nabla_X Y-\left<\nabla_X Y,\nu\right>\nu.
\]
Since $\left<Y,\nu\right>=0$ we have $\left<\nabla_X Y,\nu\right>=-\left<Y,\nabla_X \nu\right>$, and we compute
\begin{equation}
\label{eq:connection_normal}
  \left<Y,\nabla_X \nu\right>=\left<Y,2|d\rho|^{-1}\re\left(\sum_{j=1}^n\left(X\frac{\partial\rho}{\partial\bar z_j}\right)\frac{\partial}{\partial z_j}\right)\right>.
\end{equation}
For $X,Y\in T(\mathbb{C}^n)$, we define the hessian to be the operator
\[
  \Hess(X,Y)=XY-\nabla_X Y,
\]
and note that this is symmetric for the Levi-Civita connection (since $[X,Y] = \nabla_X Y - \nabla_Y X$).
With this notation, we can simplify \eqref{eq:connection_normal} and write $\left<Y,\nabla_X \nu\right>=|d\rho|^{-1}\Hess(X,Y)\rho$.  Hence, we have
\begin{equation}
\label{eq:boundary_connection}
  \nabla^b_X Y=\nabla_X Y+|d\rho|^{-1}(\Hess(X,Y)\rho)\nu.
\end{equation}
For $X,Y\in T(\bd\Om)$, we also have the boundary hessian,
\[
  \Hess^b(X,Y)=XY-\nabla^b_X Y,
\]
but using \eqref{eq:boundary_connection} we may rewrite this as
\begin{equation}
\label{eq:boundary_hessian}
  \Hess^b(X,Y)=\Hess(X,Y)-|d\rho|^{-1}(\Hess(X,Y)\rho)\nu.
\end{equation}
This gives us the expected result $\Hess^b(X,Y)\rho=0$.  For $p\in\bd\Om$ and $X\in T_p(\bd\Om)$, let $\gamma_{p,X}(t)$ denote the unique solution to the geodesic equation $\gamma_{p,X}(0)=p$, $\gamma'_{p,X}(0)=X$, and $\gamma''_{p,X}=-|d\rho|^{-1}(\Hess(\gamma'_{p,X}(t),\gamma'_{p,X}(t))\rho)\nu|_{\gamma_{p,X}(t)}$.  Then for any $C^2$ function $f$ defined on $\gamma_{p,X}$, we have
\begin{multline*}
  \frac{d^2}{dt^2}f(\gamma_{p,X}(t))=\left(\Hess(\gamma'_{p,X}(t),\gamma'_{p,X}(t))f-|d\rho|^{-1}(\Hess(\gamma'_{p,X}(t),\gamma'_{p,X}(t))\rho)\nu f\right)|_{\gamma_{p,X}(t)}\\
  =\left(\Hess^b(\gamma'_{p,X}(t),\gamma'_{p,X}(t))f\right)|_{\gamma_{p,X}(t)},
\end{multline*}
so $\Hess^b$ is the natural tool for computing convexity of quantities defined on the boundary of $\Omega$.

Let $T^{1,0}(\bd\Omega)$ be the space of $C^{m-1}$ sections of $T^{1,0}_z(\bd\Omega)$ and $T^{0,1}(\bd\Omega) = \overline{T^{1,0}(\bd\Omega)}$.
The induced CR-structure on $\bd\Omega$  at $z\in\bd\Omega$ is
\[
T^{1,0}_z(\bd\Omega)  = \{ L\in T^{1,0}(\C) : \p\rho(L)=0 \}.
\]
We denote the exterior algebra generated by these spaces by $T^{p,q}(\bd\Omega)$.
If we normalize $\rho$ so that $|d\rho|=1$ on $\bd\Omega$, then \emph{the Levi form} $\L$ is the real element of $\Lambda^{1,1}(\bd\Omega)$ defined by
\[
\opL(L,\LL) = i\p\dbar\rho(-iL\wedge\LL)
\]
for any $L\in T^{1,0}(\bd\Omega)$.
\begin{defn}\label{defn:Levi form kernel and tau3}
For $p\in\bd\Om$, let $K^{1,0}_p$ denote the \emph{kernel of the Levi form}, i.e.,
$L\in T^{1,0}_p(\bd\Om)$ is in $K^{1,0}_p$ if and only if $\opL(L, \bar L')=0$ at $p$ for all $L'\in T^{1,0}_p(\bd\Om)$.

If $L,L'\in K^{1,0}_p$ and
$X\in T_p(\bd\Om)$, we will also make use of \emph{Zaitsev's cubic invariant} \cite{Zai19} (some form of this invariant goes back at least as far as Ebenfelt's work in \cite{Ebe98})
\[
  \tau^3_p(X,L,\bar L')=X(\mathcal{L}(L,\bar L')).
\]
Note that for any fixed $X\in T_p(\bd\Om)$, $\tau^3_p(X,\cdot,\cdot)$ defines a hermitian matrix on $K^{1,0}_p$.
\end{defn}
\begin{rem}
The vector field $X$ in the definition of $\tau^3$ is a \emph{real} vector field as it is an element of $T_p(\bd\Om)$ and not
$\C T_p(\bd\Om)$.
\end{rem}

\begin{rem}
Zaitsev \cite{Zai19} actually defines $\tau^3_p(Z,L,\bar L')$ for $Z\in\mathbb{C}T_p(\bd\Om)$,
but this generality comes at the expense of the hermitian matrices that we will need,
and is easily recovered from our definition by complexifying.
\end{rem}

\begin{rem} \label{rem:notation}
Throughout the paper, we denote by $\mu_1,\dots,\mu_{n-1}$ the eigenvalues of the Levi form in increasing order. Given a point $p\in\bd\Om$, we let
$n_p^-$  be the number of negative eigenvalues of the Levi form at $p$ and $\{L_1^p,\dots, L_{n-1}^p\}$ a (local) orthonormal basis of $T_p^{1,0}(\bd\Om)$. Since the Levi form
is a Hermitian matrix, we can find an orthonormal basis of eigenvectors that we denote by $\{L_1^\mu,\dots,L_{n-1}^\mu\}$ where $L_j^\mu$ is an unit eigenvector with eigenvalue
$\mu_j$. This means that if the dimension of the null space of the Levi form at $p$ is $r$, then the null space at $p$ has orthonormal basis $\{L_{n_p^-+1}^\mu,\dots,L_{n_p^-+r}^\mu\}$.
Additionally, we denote the unit complex normal by $L_n$. This means $\{L_1^\mu,\dots, L_{n-1}^\mu, L_n\}$ forms an orthonormal
basis of $T_p^{1,0}$, and we choose $L_n$ so that the real, outward normal $\nu$ satisfies $\nu=\frac{1}{\sqrt{2}}(L_n+\bar L_n)$.
\end{rem}

Higher order invariants require greater care.  In the first place, the kernel of $\tau^3_p$ depends on the choice of
$X$.  For $X\in T_p(\bd\Om)$, we define $K^{1,0}_{p,X}$ to be the set of $L\in K^{1,0}_p$ such that $\tau^3_p(X,L,\bar L')=0$
for all $L'\in K^{1,0}_p$.
To define an invariant $\tau^4_p$ on $K^{1,0}_{p,X}$, we will need to consider our choice of coordinates.

\begin{defn}
\label{defn:tilde_k}
  Let $U_p$ be a neighborhood of $p$ on which $\mu_k\neq \mu_j$ whenever $\mu_k(p)\neq 0$ and $\mu_j(p)=0$.  Let $\tilde K^{1,0}_p(\bd\Om)\subset T^{1,0}(\bd\Om)$ denote the space of vector fields with $C^2$ coefficients on $U_p$ in the span of $\{L_j^\mu:\mu_j(p)=0\}$.
\end{defn}
Eigenvectors may not depend smoothly on $p$, so for a given point $p$, we denote by $\{L_1,\dots, L_n\}$ an orthonormal basis of $T_p^{1,0}$
so that $\{L_1,\dots,L_{n-1}\}$ is an orthonormal basis of $T^{1,0}(\bd\Om)$ with $C^2$ coefficients and
that $L_j(p)=L_j^\mu(p)$ for all $1\leq j\leq n-1$ and $L_j\in\tilde K_p^{1,0}(\bd\Om)$
whenever $L_j^\mu(p)\in K_p^{1,0}$ (see Remark \ref{rem:notation} for the notation).  Such a basis exists because we have previously shown that the orthogonal projection of $T^{1,0}(\bd\Om)$ onto $\tilde K^{1,0}_p(\bd\Om)$ is $C^2$
(Lemma A.1 in \cite{HaRa17z}).

With this structure in place, we are ready to define our fourth-order invariant.  For $X,Y\in T_p(\bd\Om)$ and $L,L'\in \tilde K^{1,0}_p(\bd\Om)$, we define
\[
  \tau^4_p(X,Y,L,\bar L')=\Hess^b(X,Y)\opL(L,\bar L').
\]
We will show in Section \ref{sec:CR_geometry} that this definition depends only on $L|_p$ and $L'|_p$ when $L,L'\in K_{p,X}^{1,0}(b\Omega)\cap K_{p,Y}^{1,0}(b\Omega)$.  This is not obvious since $\tilde K^{1,0}_p(b\Omega)$ depends on local information.

\begin{defn}\label{defn:weak Z(q)}
Let $\Omega\subset \C^n$ be a domain with a $C^m$ defining function $\rho$, $m\geq 2$. We say $\bd\Omega$ (or $\Omega$) satisfies
\emph{Z(q) weakly}
if there exists a hermitian matrix $\Upsilon=(\Upsilon^{\bar k j})$ of functions on $b\Omega$ that are uniformly bounded in $
C^{m-1}$ such that $\sum_{j=1}^{n}\Upsilon^{\bar k j}\rho_j=0$ on $\bd\Omega$ and:
\begin{enumerate}\renewcommand{\labelenumi}{(\roman{enumi})}
 \item The Eigenvalue Condition: All eigenvalues of $\Upsilon$ lie in the interval $[0,1]$.

 \item The Boundary Condition:
 $\mu_1+\cdots+\mu_q-\sum_{j,k=1}^n\Upsilon^{\bar k j}\rho_{j\bar k}\geq 0$ where  $\mu_1,\ldots,\mu_{n-1}$ are the eigenvalues of the Levi form $\L$ in increasing order.

 \item The Trace Condition: $ \inf_{z\in\bd\Omega} \{ |q-\Tr(\Upsilon)|\} >0$.

\end{enumerate}
\end{defn}
The main result of \cite{HaRa15} is that weak $Z(q)$ suffices to establish a closed range estimate for $\dbar$ at levels $q$ and
$q-1$ (and hence solvability with estimates of the equation $\dbar u=f$ in $L^2$ at these form levels). In fact, weak $Z(q)$ suffices to solve the $\dbar$-Neumann problem
on unbounded domains with estimates in specially constructed $L^2$-Sobolev spaces \cite{HaRa14} that lead to solvability in $C^\infty$ \cite{HaRa18, HaRa19}.

With these definitions in place, we are able to state our main result.
\begin{thm}
\label{thm:main_theorem}
  Let $\Omega\subset\mathbb{C}^n$ be a bounded domain with a $C^4$ boundary, and suppose that for some $1\leq q\leq n-1$ the Levi form of $\Omega$ has at most $q-1$ negative eigenvalues and at least $n-q-1$ positive eigenvalues.
 Suppose that for each $p\in\bd\Om$ at which the Levi form has exactly $n-q-1$ positive eigenvalues, there exists a $(q-n_p^-)\times(q-n_p^-)$ positive semi-definite matrix
 $N_p$ such that for all $X\in T_p(b\Omega)$
  \begin{equation}
  \label{eq:tau_3_hypothesis}
    \sum_{j,k=n_p^-+1}^q N_p^{\bar k j}\tau^3_p(X,L_j^p,\bar L_k^p)=0,
  \end{equation}
  and if there exists $X\in T_p(\bd\Om)$ such that
  \begin{equation}
  \label{eq:kernel_hypothesis}
    \sum_{j=n_p^-+1}^q N_p^{\bar k j}L_j^p\in K_{p,X}^{1,0},
  \end{equation}
  for all $n_p^-+1\leq k\leq q$, then
  \begin{equation}
  \label{eq:tau_4_hypothesis}
    \sum_{j,k=n_p^-+1}^q N_p^{\bar k j}\tau^4_p(X,X,L_j^{p},\bar L_k^{p})>0.
  \end{equation}
  Then
  \begin{enumerate}
    \item The space of harmonic $(0,q)$-forms $\mathcal{H}^q(\Omega)$ is trivial.

    \item The $\dbar$-Laplacian $\Box^q$ has closed range in $L^2_{(0,q)}(\Omega)$.

    \item The $\dbar$-Neumann operator $N^q$ exists and is continuous in $L^2_{(0,q)}(\Omega)$.

    \item The operator $\dbar$ has closed range in $L^2_{(0,q)}(\Omega)$ and $L^2_{(0,q+1)}(\Omega)$.

    \item The operator $\dbar^*$ has closed range in $L^2_{(0,q)}(\Omega)$ and $L^2_{(0,q-1)}(\Omega)$.
  \end{enumerate}
\end{thm}

\begin{rem}
  Definition \ref{defn:weak Z(q)} is relevant because we will actually show that $\Omega$ satisfies weak $Z(q)$, and the conclusions of Theorem \ref{thm:main_theorem} will follow from Theorem 1.1 in \cite{HaRa15}.
\end{rem}

\begin{rem}
\label{rem:kernel_hypothesis}
  Observe that \eqref{eq:kernel_hypothesis} and the fact that $N_p$ is hermitian guarantee that $\tau^4_p(X,X,\cdot,\cdot)$ is only evaluated on $K_{p,X}^{1,0}\times K_{p,X}^{0,1}$ in \eqref{eq:tau_4_hypothesis}.
\end{rem}

We will see in the proof of Theorem \ref{thm:main_theorem} that the hypotheses of the theorem are sufficiently strong to guarantee that the degenerate points (i.e., those with exactly $n-q-1$ positive eigenvalues) are isolated.  Although the hypotheses of Theorem \ref{thm:main_theorem} are still quite complicated, we note that they have some significant advantages over Definition \ref{defn:weak Z(q)}.  In particular, Definition \ref{defn:weak Z(q)} requires that an appropriate $\Upsilon$ be found out of the class of matrices of $C^2$ functions, while Theorem \ref{thm:main_theorem} only requires that $N_p$ be found in the class of constant-coefficient matrices.  This reduction to a finite-dimensional space introduces the possibility that linear algebra can be used to construct $N_p$.

\subsubsection{Linear Algebra}

Motivated by Theorem \ref{thm:main_theorem}, we now introduce the necessary linear algebra to simplify the hypotheses.  We equip the $\mathbb{R}$-linear space of hermitian $n\times n$ matrices with the real inner product $\left<A,B\right>=\Tr(A B)$.  Note that this inner product is invariant under unitary changes of coordinates, so we may assume that either $A$ or $B$ is diagonal.  Hence, $|A|^2=\left<A,A\right>$ will equal the sum of the squares of the eigenvalues of $A$, so this induces a norm on the space of $n\times n$ hermitian matrices.  Note that this is the $\ell^2$ norm discussed in Section 5.6 of \cite{HoJo85}.  Since $|A|$ is at least as large as the size of the largest eigenvalue of $A$, we also
know $|Av|\leq|A||v|$ for any $n\times 1$ vector $v$.

We will often speak of restricting an $n\times n$ hermitian matrix $M$ to an $m$-dimensional subspace $V$.  By this, we mean the restriction of the quadratic form induced by $M$ to vectors in $V$.  We will often use the fact that if $P_V$ is an $n\times m$ matrix with columns forming an orthonormal basis for $V$, then $M|_V$ can be identified with the matrix $\bar P_V^T M P_V$.

Our fundamental structure will be the following:
\begin{defn}
\label{defn:indefinite}
  Let $\{M_j\}$ be a collection of hermitian $n\times n$ matrices, $n\geq 1$.  We say that $\Span_{\mathbb{R}}\{M_j\}$ is \emph{indefinite} if there does not exist a real linear combination of $\{M_j\}$ that is positive definite.  If $V\subset\mathbb{C}^n$ is a nontrivial vector space, we say that $\Span_{\mathbb{R}}\{M_j\}$ is \emph{indefinite on $V$} if there does not exist a real linear combination of $\{M_j\}$ that is positive definite when restricted to $V$.
\end{defn}
\begin{rem}
Observe that if $V'\subset V$ and $\Span_{\mathbb{R}}\{M_j\}$ is indefinite on $V'$, then $\Span_{\mathbb{R}}\{M_j\}$ is indefinite on $V$.
\end{rem}

\begin{rem}
Saying that $\{M_j\}$ is indefinite on $V$ is equivalent to the statement that $\bar P_V^T M P_V$ is indefinite on $\C^m$, where $m=\dim V$.
\end{rem}

\begin{defn}
\label{defn:minimal}
  Let $\{M_j\}$ be a collection of hermitian $n\times n$ matrices, $n\geq 1$.  We say that a nontrivial vector space $V\subset\mathbb{C}^n$ is
  \emph{minimal with respect to $\Span_{\mathbb{R}}\{M_j\}$} if $\Span_{\mathbb{R}}\{M_j\}$ is indefinite on $V$ but $\Span_{\mathbb{R}}\{M_j\}$ is not indefinite on any nontrivial proper subspace of $V$.
\end{defn}
If $\Span_{\mathbb{R}}\{M_j\}$ is indefinite, then a minimal subspace with respect to $\Span_{\mathbb{R}}\{M_j\}$ must exist, although it may equal $\mathbb{C}^n$.  Minimal subspaces are not necessarily unique, however, as we will see.

With the linear algebra in place, we can state the following two consequences of Theorem \ref{thm:main_theorem}.
\begin{thm}
\label{thm:isolated degeneracy}
  Let $\Omega\subset\mathbb{C}^n$ be a bounded domain with a $C^4$ boundary, and suppose that for some $1\leq q\leq n-1$ the Levi form of $\Omega$ has at most $q-1$ negative eigenvalues and at least $n-q-1$ positive eigenvalues.  Suppose that for each $p\in\bd\Om$ at which the Levi form has exactly $n-q-1$ positive eigenvalues and $X\in T_p(\bd\Om)$ for which $K^{1,0}_{p,X}$ is nontrivial, if $\Span_{Y\in T_p(\bd\Om)}\tau^3_p(Y,\cdot,\cdot)$ is indefinite on $K^{1,0}_{p,X}$ then there exists $Y\in T_p(\bd\Om)$ such that $\tau^3_p(Y,\cdot,\cdot)+\tau^4_p(X,X,\cdot,\cdot)$ is positive definite on $K^{1,0}_{p,X}$.  Then
  \begin{enumerate}
    \item The space of harmonic $(0,q)$-forms $\mathcal{H}^q(\Omega)$ is trivial.

    \item The $\dbar$-Laplacian $\Box^q$ has closed range in $L^2_{(0,q)}(\Omega)$.

    \item The $\dbar$-Neumann operator $N^q$ exists and is continuous in $L^2_{(0,q)}(\Omega)$.

    \item The operator $\dbar$ has closed range in $L^2_{(0,q)}(\Omega)$ and $L^2_{(0,q+1)}(\Omega)$.

    \item The operator $\dbar^*$ has closed range in $L^2_{(0,q)}(\Omega)$ and $L^2_{(0,q-1)}(\Omega)$.
  \end{enumerate}
\end{thm}

The proof of
Theorem \ref{thm:isolated degeneracy} does not make full use of the structures introduced in Section \ref{sec:indefinite_subspaces}.
This makes the hypotheses relatively easy to check but excludes some examples.
The slightly more complicated and general theorem below can be proven using the same technique.

\begin{thm}
\label{thm:isolated degeneracy_improved}
Let $\Omega\subset\mathbb{C}^n$ be a bounded domain with a $C^4$ boundary,
and suppose that for some $1\leq q\leq n-1$ the Levi form of $\Omega$ has at most $q-1$ negative eigenvalues and at least $n-q-1$
positive eigenvalues.
Suppose that for each $p\in\bd\Om$ at which the Levi form has exactly $n-q-1$ positive eigenvalues there exists a nontrivial
subspace $V_p\subset K^{1,0}_p$ that is minimal with respect to
$\Span_{Y\in T_p(\bd\Om)}\tau^3_p(Y,\cdot,\cdot)$ with the property that for every $X\in T_p(\bd\Om)$,
if $V_p\subset K^{1,0}_{p,X}$ then there exists $Y\in T_p(\bd\Om)$ such that $\tau^3_p(Y,\cdot,\cdot)+\tau^4_p(X,X,\cdot,\cdot)$
is positive definite on $V_p$.  Then
  \begin{enumerate}
    \item The space of harmonic $(0,q)$-forms $\mathcal{H}^q(\Omega)$ is trivial.

    \item The $\dbar$-Laplacian $\Box^q$ has closed range in $L^2_{(0,q)}(\Omega)$.

    \item The $\dbar$-Neumann operator $N^q$ exists and is continuous in $L^2_{(0,q)}(\Omega)$.

    \item The operator $\dbar$ has closed range in $L^2_{(0,q)}(\Omega)$ and $L^2_{(0,q+1)}(\Omega)$.

    \item The operator $\dbar^*$ has closed range in $L^2_{(0,q)}(\Omega)$ and $L^2_{(0,q-1)}(\Omega)$.
  \end{enumerate}
\end{thm}

%
%
\section{Proof of Theorem \ref{thm:main_theorem}}
\label{sec:main_proofs}

\subsection{CR Geometry}
\label{sec:CR_geometry}

Before we prove the main theorems, we go through the CR geometry to show that $\tau^4$ is indeed a tensor.
In what follows, we use the notation of Remark \ref{rem:notation} and the subsequent paragraph.

For $X\in T(\bd\Om)$, let $\Gamma_{X,j}^k=\left<\nabla_X L_j,L_k\right>$ when $1\leq j,k\leq n$.  Since $\mathbb{C}^n$ is K\"ahler under the Euclidean metric, $\left<\nabla_X L_j,\bar L_k\right>=0$.  With this notation, we also have
\[
  \nabla_X^b L_j=\sum_{k=1}^{n-1}\Gamma_{X,j}^k L_k+\frac{1}{2}\Gamma_{X,j}^n(L_n-\bar L_n).
\]
We may differentiate the equality $\partial\rho(L_j)=|\partial\rho|\left<L_j,L_n\right>$
by $X$ (noting that $\left<L_j,L_n\right>$ is constant, so $X(|\partial\rho|\left<L_j,L_n\right>)=(X|\partial\rho|)\left<L_j,L_n\right>$) to obtain

\[
 (\nabla_X\partial\rho)(L_j)+|\partial\rho|\Gamma_{X,j}^n=(X|\partial\rho|)\la L_j,L_n\ra.
\]
Hence,
\[
  \Gamma_{X,j}^n=-|\partial\rho|^{-1}(\nabla_X\partial\rho)(L_j)+(X\log|\partial\rho|)\left<L_j,L_n\right>
\]
whenever $1\leq j\leq n$.  We define a tensor on $T(\bd\Om)\times T^{1,0}(\mathbb{C}^n)$ by
\[
  \theta^n(X,L)=-|\partial\rho|^{-1}(\nabla_X\partial\rho)(L)+(X\log|\partial\rho|)\la L,L_n\ra,
\]
so that $\Gamma_{X,j}^n=\theta^n(X,L_j)$.  Since $\left<L_j,L_k\right>$ is constant, we may differentiate this by $X$ to obtain
\begin{equation}
\label{eq:skew_symmetry}
  0=\Gamma_{X,j}^k+\overline{\Gamma_{X,k}^j},
\end{equation}
so the matrix with coefficients $\Gamma_{X,j}^k$ is skew-hermitian.

We can extend $\opL$ to $T^{1,0}(\mathbb{C}^n)\times T^{0,1}(\mathbb{C}^n)$ by requiring $\opL(L_n,\bar L_j)=0$ for all $1\leq j\leq n$.  With this convention, we have for any $X\in T(\bd\Om)$ and $1\leq j,k\leq n$
\begin{equation}
\label{eq:levi_form_derivative}
  X(\opL(L_j,\bar L_k))=(\nabla_X^b\opL)(L_j,\bar L_k)
  +\sum_{\ell=1}^{n-1}\Gamma_{X,j}^\ell\opL(L_\ell,\bar L_k)+\sum_{\ell=1}^{n-1}\overline{\Gamma_{X,k}^\ell}\opL(L_j,\bar L_\ell).
\end{equation}

Hence, if $L,L'\in K_p^{1,0}$, we have
\[
  \tau^3_p(X,L,\bar L')=(\nabla_X^b\opL)(L,\bar L').
\]
Furthermore, for any $1\leq j\leq n$, since $\opL(L_j,\bar L_n)\equiv 0$, \eqref{eq:skew_symmetry} and \eqref{eq:levi_form_derivative} give us
\[
  0=(\nabla_X^b\opL)(L_j,\bar L_n)-\sum_{\ell=1}^{n-1}\theta^n(X,L_\ell)\opL(L_j,\bar L_\ell),
\]
so
\begin{equation}
\label{eq:normal_derivative}
  (\nabla_X^b\opL)(L_j,\bar L_n)=\sum_{\ell=1}^{n-1}\theta^n(X,L_\ell)\opL(L_j,\bar L_\ell).
\end{equation}

By assumption, if $1\leq j,k\leq n-1$ satisfy $\mu_j(p)=0$ and $\mu_k(p)\neq 0$, then $L_j\in\tilde K_p^{1,0}(\bd\Om)$ and $L_k$ is orthogonal to $\tilde K_p^{1,0}(\bd\Om)$, so $\opL(L_j,\bar L_k)\equiv 0$ and evaluating \eqref{eq:levi_form_derivative} at $p$ gives us
\[
  0=(\nabla_X^b\opL)(L_j,\bar L_k)+\Gamma_{X,j}^k\mu_k(p).
\]
Hence, we may define a family of tensors on $T(\bd\Om)\times T^{1,0}(\mathbb{C}^n)$ by
\[
  \theta^k(X,L)=
  -(\mu_k(p))^{-1}(\nabla_X^b\opL)(L,\bar L_k)
\]
and obtain $\Gamma_{X,j}^k=\theta^k(X,L_j)$ whenever $\mu_j(p)=0$ and $\mu_k(p)\neq 0$.  The existence of this tensor is the most important consequence of requiring $L,L'\in\tilde K^{1,0}_p(\bd\Om)$ in the definition of $\tau^4_p$.

To assist in taking second derivatives of $\opL$, we let $S_1$ denote the set of $1\leq \ell\leq n-1$ such that $\mu_\ell(p)=0$ and $S_2$ denote the set of $1\leq \ell\leq n-1$ such that $\mu_\ell(p)\neq 0$.  Fix $X,Y\in T(\bd\Om)$ and $L_j,L_k\in\tilde K^{1,0}_p(\bd\Om)$.  Note that in this case, \eqref{eq:normal_derivative} gives us $(\nabla_X^b\opL)(L_j,\bar L_n)|_p=0$.  With this in mind, differentiating \eqref{eq:levi_form_derivative} again gives us
\begin{multline*}
  YX(\opL(L_j,\bar L_k))|_p=(\nabla_Y^b\nabla_X^b\opL)(L_j,\bar L_k)+\sum_{\ell\in S_1}\Gamma_{Y,j}^\ell\tau^3_p(X,L_\ell,\bar L_k)\\
  +\sum_{\ell\in S_2}\theta^\ell(Y,L_j)(\nabla_X^b\opL)(L_\ell,\bar L_k)+\sum_{\ell\in S_1}\overline{\Gamma_{Y,k}^\ell}\tau^3_p(X,L_j,\bar L_\ell)+\sum_{\ell\in S_2}\overline{\theta^\ell(Y,L_k)}(\nabla_X^b\opL)(L_j,\bar L_\ell)\\
  +\sum_{\ell\in S_1}\Gamma_{X,j}^\ell\tau^3_p(Y,L_\ell,\bar L_k)+\sum_{\ell\in S_2}\theta^\ell(X,L_j) (\nabla^b_Y\opL)(L_\ell,\bar L_k)+\sum_{\ell\in S_2}\mu_\ell(p)\theta^\ell(X,L_j)\overline{\theta^\ell(Y,L_k)}\\
  +\sum_{\ell\in S_1}\overline{\Gamma_{X,k}^\ell}\tau^3_p(Y,L_j,\bar L_\ell)
  +\sum_{\ell\in S_2}\overline{\theta^\ell(X,L_k)}(\nabla^b_Y\opL)(L_j,\bar L_\ell)
  +\sum_{\ell\in S_2}\mu_\ell(p)\theta^\ell(Y,L_j)\overline{\theta^\ell(X,L_k)}.
\end{multline*}
Using this, we may show that for any $L,L'\in K^{1,0}_{p,X}\cap K^{1,0}_{p,Y}$, we have
\begin{multline*}
  \tau^4_p(Y,X,L,\bar L')=(\nabla_Y^b\nabla_X^b\opL)(L,\bar L')
  +\sum_{\ell\in S_2}\theta^\ell(Y,L)(\nabla_X^b\opL)(L_\ell,\bar L')\\
  +\sum_{\ell\in S_2}\overline{\theta^\ell(Y,L')}(\nabla_X^b\opL)(L,\bar L_\ell)
  +\sum_{\ell\in S_2}\theta^\ell(X,L) (\nabla^b_Y\opL)(L_\ell,\bar L')+\sum_{\ell\in S_2}\mu_\ell(p)\theta^\ell(X,L)\overline{\theta^\ell(Y,L')}\\
  +\sum_{\ell\in S_2}\overline{\theta^\ell(X,L')}(\nabla^b_Y\opL)(L,\bar L_\ell)+\sum_{\ell\in S_2}\mu_\ell(p)\theta^\ell(Y,L)\overline{\theta^\ell(X,L')}-\tau^3_p(\nabla^b_Y X,L,\bar L').
\end{multline*}
Hence, even though our definition of $\tau^4_p$ used the local behavior of $L$ and $L'$, the end result is a tensor depending only on the values of $L$ and $L'$ at $p$ and we have therefore proved the following proposition.

\begin{prop}\label{prop:tau4 is a tensor}
For $p\in\partial\Omega$ and $X,Y\in T_p(b\Omega)$, $\tau^4_p(X,Y,\cdot,\cdot)$ is a tensor on
$(K^{1,0}_{p,X}\cap K^{1,0}_{p,Y})\times(K^{0,1}_{p,X}\cap K^{0,1}_{p,Y})$.
\end{prop}

We need the following lemma to prove the main theorems.
\begin{lem}
\label{lem:counting_negative_eigenvalues}
  Let $\Omega\subset\mathbb{C}^n$ be a bounded domain with $C^4$ boundary.  For any $p\in\bd\Omega$
  and $X, Y\in T_p(\bd\Omega)$, denote the number of negative eigenvalues of $\tau^3_p(X,\cdot,\cdot)$ on $K^{1,0}_p$ by
  $n_{p,X}^-$  and of  $\tau^4_p(X,X,\cdot,\cdot)+\tau^3_p(Y,\cdot,\cdot)$ on $K^{1,0}_{p,X}$ by $n_{p,X,Y}^-$.
  For $1\leq q\leq n-1$, if $n_p^-\leq q-1$ for all $p\in\bd\Omega$, then $n_p^-+n_{p,X}^-+n_{p,X,Y}^-\leq q-1$ for all $p\in\bd\Omega$ and $X,Y\in T_p(\bd\Omega)$.
\end{lem}

\begin{proof}
  For $p\in\bd\Omega$, let $X,Y\in T_p(\bd\Omega)$.  We write
  \[
    X=\sum_{j=1}^n 2\re\left(a^j\frac{\partial}{\partial z_j}\right), Y=\sum_{j=1}^n 2\re\left(b^j\frac{\partial}{\partial z_j}\right),
    \text{ and } \nabla_X X - \nabla^b_X X=\sum_{j=1}^n 2\re\left(c^j\frac{\partial}{\partial z_j}\right).
  \]
  For $t\in\mathbb{R}$, set $\tilde\gamma_t^j=p_j+t a^j+\frac{1}{2}t^2(b^j+c^j)$.  If $\rho$ is a $C^4$ defining function for $\Omega$, we have $|\rho(\tilde\gamma_t)|\leq O(t^3)$, so by projecting $\tilde\gamma_t$ onto the nearest boundary point we obtain a $C^3$ curve $\gamma_t\subset\bd\Omega$ for sufficiently small $t$ such that $|\gamma_t-\tilde\gamma_t|\leq O(t^3)$.  Furthermore, if $f$ is a $C^2$ function on some neighborhood of $p$ in $\bd\Omega$, we have $\frac{d}{dt}f(\gamma_t)|_{t=0}=(Xf)|_p$ and
 $\frac{d^2}{dt^2}f(\gamma_t)|_{t=0}=((\Hess^b(X,X)+Y)f)|_p$.

Let $\mu_j$ and $L_j^\mu$ be as in Definition \ref{defn:tilde_k}.  Let $\{L_1,\ldots,L_{n-1}\}$ be an orthonormal basis for $T^{1,0}(\bd\Omega)$
with $C^2$ coefficients such that $L_j|_p=L_j^\mu|_p$ and $L_j\in\tilde K_p^{1,0}(\bd\Omega)$ whenever $\mu_j(p)=0$.
After an orthonormal change of coordinates preserving these properties, we may assume that $\tau^3_p(X,L_j^\mu|_p,\bar L_k^\mu|_p)$ is
diagonal on $K^{1,0}_p$ with increasing entries on the diagonal, and $\tau^4_p(X,X,L_j^\mu|_p,\bar L_k^\mu|_p)+\tau^3_p(Y,L_j^\mu|_p,\bar L_k^\mu|_p)$
is diagonal on $K^{1,0}_{p,X}$ with increasing entries on the diagonal.

 Let $m=n_p^-+n_{p,X}^-+n_{p,X,Y}^-$.  We claim that $\opL$ is negative definite on $\Span_{\mathbb{C}}\{L_1,\ldots,L_m\}$ at $\gamma_t$ for all $t\neq 0$ sufficiently small.  Suppose, on the contrary, that for every $\ell\in\mathbb{N}$ there exists $0<t_\ell<1/\ell$ and a vector $\zeta_\ell\in\mathbb{C}^m$ such that $\opL(Z_\ell,\bar Z_\ell)\geq 0$
 at $\gamma_{t_\ell}$ when $Z_\ell=\sum_{j=1}^m \zeta^j_\ell L_j|_{\gamma_{t_\ell}}$.  We may normalize so that $|\zeta_\ell|\equiv 1$.  Note that $Z_\ell$ extends to a neighborhood of $p$ by keeping the coefficients $\zeta_\ell^j$ constant with respect to the basis $\{L_j\}$.  Since each $L_j$ is $C^2$ and $|\zeta_\ell|$ is uniformly bounded, $Z_\ell$ will have a uniformly bounded $C^2$ norm.  By construction of the basis $\{L_j\}$,
  \[
    \opL(Z_\ell,\bar Z_\ell)|_{\gamma_{t_\ell}}
    =\sum_{j,k=1}^{n_p^-}\zeta^j_\ell\bar \zeta^k_\ell\opL(L_j,\bar L_k)+\sum_{j,k=n_p^-+1}^{m}\zeta^j_\ell\bar \zeta^k_\ell\opL(L_j,\bar L_k).
  \]
  Since the first sum is strictly negative if $\ell$ is sufficiently large, we may assume that $\zeta^j_\ell=0$ whenever $1\leq j\leq n_p^-$.  Restricted to these coordinates, a Taylor expansion of $\opL(Z_\ell,\bar Z_\ell)|_{\gamma_{t_\ell}}$ as a function of $t$, combined with the construction
of $\gamma_t$ and the vectors $\{L_j\}$ shows that
  \[
    \opL(Z_\ell,\bar Z_\ell)|_{\gamma_{t_\ell}}=t_\ell\tau^3_p(X,Z_\ell|_p,\bar Z_\ell|_p)+O(t_\ell^2)
    =t_\ell\sum_{j=n_p^-+1}^{n_p^-+n_{p,X}^-}|\zeta_\ell^j|^2\tau^3_p(X,L_j^\mu,\bar L_j^\mu)+O(t_\ell^2).
  \]
  Since $\tau^3_p(X,L_j^\mu,\bar L_j^\mu)<0$ when $n_p^-+1\leq j\leq n_p^-+n_{p,X}^-$, we must have $|\zeta_\ell^j|\leq O(\sqrt{t_\ell})$ whenever $n_p^-+1\leq j\leq n_p^-+n_{p,X}^-$ and $\ell$ is sufficiently large.  Hence, because of our choice of coordinates,
  \begin{multline*}
    \opL(Z_\ell,\bar Z_\ell)|_{\gamma_{t_\ell}}
    =t_\ell\tau^3_p(X,Z_\ell|_p,\bar Z_\ell|_p)+t_\ell^2\big(\tau^4_p(X,X,Z_\ell|_p,\bar Z_\ell|_p)+\tau^3_p(Y,Z_\ell|_p,\bar Z_\ell|_p)\big)+o(t_\ell^2)\\
    =t_\ell\sum_{j=n_p^-+1}^{n_p^-+n_{p,X}^-}|\zeta_\ell^j|^2\tau^3_p(X,L_j^\mu,\bar L_j^\mu)\\
    +t_\ell^2\sum_{j=n_p^-+n_{p,X}^-+1}^{m}|\zeta^\ell_j|^2\big(\tau^4_p(X,X,L_j^\mu,\bar L_j^\mu)+\tau^3_p(Y,L_j^\mu,\bar L_j^\mu)\big)+o(t_\ell^2).
  \end{multline*}
  In order for this to be positive, we must have $|\zeta_\ell^j|\rightarrow 0$ whenever $n_p^-+n_{p,X}^-+1\leq j\leq m$, and hence $\zeta^\ell\rightarrow 0$, contradicting the fact that $|\zeta_\ell|\equiv 1$.

  Now, we know that the Levi form must have at least $m$ negative eigenvalues on $\gamma_t$ for $t$ sufficiently small.  By assumption, we must have $m\leq q-1$, and our conclusion follows.
\end{proof}

\begin{cor}
\label{cor:counting_negative_eigenvalues}
  Let $\Omega\subset\mathbb{C}^n$ be a bounded domain with a $C^4$ boundary, and suppose that for some $1\leq q\leq n-1$ and every $p\in b\Omega$ the Levi form of $\Omega$ at $p$ has at most $q-1$ negative eigenvalues and at least $n-q-1$ positive eigenvalues.  If the Levi form of $\Omega$ has exactly $n-q-1$ positive eigenvalues at some $p\in\bd\Omega$, then $\Span_{Y\in T_p(\bd\Omega)}\tau^3_p(Y,\cdot,\cdot)$ is indefinite on $K^{1,0}_p$.  At such a $p$, if there exists $X\in T_p(\bd\Omega)$ such that $\tau^3_p(X,\cdot,\cdot)$ is positive semi-definite on $K^{1,0}_p\times K^{0,1}_p$, then $\Span_{Y\in T_p(\bd\Omega)}\tau^3_p(Y,\cdot,\cdot)$ is indefinite on $K^{1,0}_{p,X}$ and either
  \[
    \Span_{Y\in T_p(b\Omega)}\tau^3_p(Y,\cdot,\cdot) + \tau^4_p(X,X,\cdot,\cdot)
  \]
is indefinite on $K^{1,0}_{p,X}$ or there exists $Y\in T_p(\bd\Omega)$ such that $\tau_p^3(Y,\cdot,\cdot)+\tau_p^4(X,X,\cdot,\cdot)$ is positive
definite on $K^{1,0}_{p,X}\times K^{0,1}_{p,X}$.
\end{cor}

\begin{proof}
  Since the dimension of $K^{1,0}_p$ is at least $1$, Lemma \ref{lem:counting_negative_eigenvalues} implies that $\tau^3_p(X,\cdot,\cdot)$ is not negative definite on $K^{1,0}_p$ for any $X\in T_p(\bd\Omega)$.  Since $\tau^3_p(X,\cdot,\cdot)=-\tau^3_p(-X,\cdot,\cdot)$, $\tau^3_p(X,\cdot,\cdot)$ is also never positive definite, and hence $\Span_{Y\in T_p(\bd\Omega)}\tau^3_p(Y,\cdot,\cdot)$ is indefinite on $K^{1,0}_p$.

If there exists $X\in T_p(\bd\Omega)$ such that $\tau^3_p(X,\cdot,\cdot)$ is positive semi-definite on $K^{1,0}_p\times K^{0,1}_p$, then
Lemma \ref{lem:semi_definite_direction} guarantees that $\Span_{Y\in T_p(\bd\Omega)}\tau^3_p(Y,\cdot,\cdot)$ is indefinite on $K^{1,0}_{p,X}$.
Since $\tau^3_p(-X,\cdot,\cdot)$ is negative semi-definite and $\tau^4_p(X,X,\cdot,\cdot)=\tau^4_p(-X,-X,\cdot,\cdot)$,
Lemma \ref{lem:counting_negative_eigenvalues} guarantees that $\tau^4_p(X,X,\cdot,\cdot)+\tau^3_p(Y,\cdot,\cdot)$
is not negative definite on $K^{1,0}_{p,X}\times K^{0,1}_{p,X}$ for any $Y\in T_p(\bd\Omega)$, leaving only the two possibilities admitted in the statement of the Corollary.
\end{proof}

\subsection{Proof of Theorem \ref{thm:main_theorem}}

\begin{proof}[Proof of Theorem \ref{thm:main_theorem}]
We continue using the notation of Section \ref{sec:results}.
Let $N_p$ be the $(q-n_p^-)\times(q-n_p^-)$ positive semi-definite matrix satisfying the hypotheses of Theorem \ref{thm:main_theorem}.
Extend $N_p$ to an $(n-1)\times(n-1)$ matrix by setting $N_p^{\bar k j}=0$ whenever $j\leq n_p^-$, $j\geq q+1$, $k\leq n_p^-$, or $k\geq q+1$.
Let $r=\rank(N_p)$.  By an orthonormal change of coordinates in $\{L_{n_p^-+1},\ldots, L_q\}$, we may assume that $N_p$ is diagonal and
$N_p^{\bar j j}=0$ whenever $n_p^-+r+1\leq j\leq q$.  Since $N_p$ is positive definite on its range,
there exists $\lambda>0$ such that for any $u\in\mathbb{C}^{n-1}$,
  \begin{equation}
  \label{eq:N_lower_bound}
    \sum_{j,k=1}^{n-1}\bar u_k N^{\bar k j}_p u_j\geq\lambda\sum_{j=n_p^-+1}^{n_p^-+r}|u_j|^2.
  \end{equation}

  Choose orthonormal coordinates so that $p=0$ and $L_j|_p=\frac{\partial}{\partial z_j}$, and write $z_j=x_j+i y_j$.  In these coordinates, we define a smooth family of $(q-n_p^-)\times(q-n_p^-)$ matrices by
  \[
    M^{\bar k j}=\sum_{\ell=1}^{n-1}x_\ell\tau_p^3\left(\frac{\partial}{\partial x_\ell},\frac{\partial}{\partial z_k},\frac{\partial}{\partial\bar z_j}\right)
    +\sum_{\ell=1}^{n}y_\ell\tau_p^3\left(\frac{\partial}{\partial y_\ell},\frac{\partial}{\partial z_k},\frac{\partial}{\partial\bar z_j}\right)
  \]
  when $n_p^-+1\leq j,k\leq q$.  Observe that $M$ has the property that $M^{\bar k j}|_p=0$ and $X M^{\bar k j}|_p=\tau_p^3(X,L_k,\bar L_j)$.

  For $t>0$, we define an extension $N_t$ of $N_p$ to a neighborhood of $p$ by setting
  \[
    N_t^{\bar k j}=\begin{cases}
      N_p^{\bar k j}+t M^{\bar k j},&n_p^-+1\leq j,k\leq q\text{ and either }j\leq n_p^-+r\text{ or }k\leq n_p^-+r,\\
      \sum_{\ell=n_p^-+r+1}^q \frac{2t^2}{\lambda}M^{\bar k\ell}M^{\bar\ell j},&n_p^-+r+1\leq j,k\leq q,\\
      0,&\text{otherwise}.
    \end{cases}
  \]
  Given $u\in\mathbb{C}^{n-1}$, we define $v\in\mathbb{C}^{n-1}$ by
  \[
    v_j=\begin{cases}
      \sum_{\ell=n_p^-+r+1}^q M^{\bar j\ell}u_\ell,&n_p^-+r+1\leq j\leq q,\\
      u_j,&\text{otherwise}.
    \end{cases}
  \]
  Using \eqref{eq:N_lower_bound},
  \begin{multline*}
    \sum_{j,k=1}^{n-1}\bar u_k N^{\bar k j}_tu_j\geq\lambda\sum_{j=n_p^-+1}^{n_p^-+r}|v_j|^2+\sum_{j,k=n_p^-+1}^{n_p^-+r} t\bar v_k M^{\bar k j}v_j\\
    +\sum_{k=n_p^-+r+1}^q\sum_{j=n_p^-+1}^{n_p^-+r}2t\re(v_k \bar v_j)+\sum_{j=n_p^-+r+1}^q 2\lambda^{-1}t^2|v_j|^2,
  \end{multline*}
  so the Cauchy-Schwarz inequality and the elementary inequality $2ab \leq \frac{\lambda}{2} a^2 + \frac{2}{\lambda}b^2$ for any real $a,b$ shows
  \[
    \sum_{j,k=1}^{n-1}\bar u_k N^{\bar k j}_tu_j
    \geq\frac{\lambda}{2}\sum_{j=n_p^-+1}^{n_p^-+r}|v_j|^2+\sum_{j,k=n_p^-+1}^{n_p^-+r} t\bar v_k M^{\bar k j}v_j.
  \]
Hence, $N_t$ is positive semi-definite on a neighborhood of $p$ that depends on $t$.  Furthermore, properties of $M$ give us
$N_t^{k\bar j}|_p=N_p^{k\bar j}$ and
\begin{equation}
\label{eq:X_of_N_t}
  X N_t^{k\bar j}|_p=t\tau^3_p(X,L_k,\bar L_j)
\end{equation}
whenever $X\in T_p(\partial\Omega)$, $n_p^-+1\leq j,k\leq q$, and either $j\leq n_p^-+r$ or $k\leq n_p^-+r$.

  Since the range of $N_p$ lies in the kernel of the Levi-form,
  \[
    \sum_{j,k=1}^{n-1}N_t^{\bar k j}\opL(L_j,\bar L_k)\Big|_p=0.
  \]
  For any $X\in T_p(\partial\Omega)$, \eqref{eq:tau_3_hypothesis} implies
  \begin{equation}\label{eqn:X Nt =0}
    X\Big(\sum_{j,k=1}^{n-1}N_t^{\bar k j}\opL(L_j,\bar L_k)\Big)\Big|_p=\sum_{j,k=n_p^-}^{q}N_p^{\bar k j}\tau^3_p(X,L_j,\bar L_k)=0.
  \end{equation}
  Furthermore,
  \begin{multline*}
    \Hess^b(X,X)\left(\sum_{j,k=1}^{n-1}N_{t}^{\bar k j}\opL(L_j,\bar L_k)\right)\Big|_p=\\
    \sum_{j,k=n_p^-+1}^{q}2\big(X N_{t}^{\bar k j}\big|_p\big)\tau^3_p(X,L_j,\bar L_k)
    +\sum_{j,k=n_p^-+1}^q N^{\bar k j}_p\Hess^b(X,X)\opL(L_j,\bar L_k),
  \end{multline*}
  so \eqref{eq:X_of_N_t} gives us
  \begin{multline*}
    \Hess^b(X,X)\left(\sum_{j,k=1}^{n-1}N_{t}^{\bar k j}\opL(L_j,\bar L_k)\right)\Big|_p=\\
    \sum_{\{n_p^-+1\leq j,k\leq q:j\leq n_p^-+r\text{ or }k\leq n_p^-+r\}}2t|\tau^3_p(X,L_j,\bar L_k)|^2
    +\sum_{j=n_p^-+1}^{n_p^-+r} N^{\bar j j}_p\Hess^b(X,X)\opL(L_j,\bar L_j).
  \end{multline*}

  Choose $\eps>0$ sufficiently small so that every eigenvalue of $\eps N_p$ is strictly less than $1$, and define an $(n-1)\times(n-1)$ matrix $\Upsilon$ by
  \[
    \Upsilon^{\bar k j}=\begin{cases}
      I_{jk},&1\leq j,k\leq n_p^-\\
      I_{jk}-\eps N^{\bar k j}_t,&n_p^-+1\leq j,k\leq q\\
      0,&\text{otherwise.}
    \end{cases}
  \]
  Note that by definition, $\sum_{j=1}^q\mu_j=\sum_{j=1}^q\mathcal{L}(L_j,\bar L_k)$.  Then
  \[
    \left(\sum_{j=1}^q\mu_j-\sum_{j,k=1}^{n-1}\Upsilon^{\bar k j}\opL(L_j,\bar L_k)\right)\bigg|_p=0,
  \]
  and for any $X\in T_p(\partial\Omega)$ we have, as a consequence of \eqref{eqn:X Nt =0}, that
  \[
    X\left(\sum_{j=1}^q\mu_j-\sum_{j,k=1}^{n-1}\Upsilon^{\bar k j}\opL(L_j,\bar L_k)\right)\bigg|_p=0.
  \]
  Moreover,
  \begin{multline}
  \label{eq:upsilon_hessian}
    \Hess^b(X,X)\left(\sum_{j=1}^q\mu_j-\sum_{j,k=1}^{n-1}\Upsilon^{\bar k j}\opL(L_j,\bar L_k)\right)\bigg|_p=\\
    \sum_{\{n_p^-+1\leq j,k\leq q:j\leq n_p^-+r\text{ or }k\leq n_p^-+r\}}2t\eps|\tau^3_p(X,L_j,\bar L_k)|^2
    +\sum_{j=n_p^-+1}^{n_p^-+r} \eps N^{\bar j j}_p\Hess^b(X,X)\opL(L_j,\bar L_j).
  \end{multline}
  Let $S\subset T_p(\partial\Omega)$ denote the set of unit-length tangent vectors.  If we can show
  \begin{equation}
  \label{eq:upsilon_positivity}
    \Hess^b(X,X)\left(\sum_{j=1}^q\mu_j-\sum_{j,k=1}^{n-1}\Upsilon^{\bar k j}\opL(L_j,\bar L_k)\right)\bigg|_p>0
  \end{equation}
  for all $X\in S$, then it will follow that \eqref{eq:upsilon_positivity} holds for all nontrivial $X\in T_p(\partial\Omega)$, and hence
  \[
    \sum_{j=1}^q\mu_j-\sum_{j,k=1}^{n-1}\Upsilon^{\bar k j}\opL(L_j,\bar L_k)>0
  \]
  on some neighborhood of $p$ (excluding $p$ itself).  This suffices to show that weak $Z(q)$ is satisfied in a neighborhood of $p$, and since weak $Z(q)$ is a local condition (on bounded domains), we conclude that weak $Z(q)$ is satisfied on $\Omega$.

Let $S_0\subset S$ denote the set of $X\in T_p(\partial\Omega)$ for which $\tau^3_p(X,L_j,\bar L_k)=0$ for all $n_p^-+1\leq j,k\leq q$ such that either
$j\leq n_p^-+r$ or $k\leq n_p^-+r$.  Then $L_j\in K^{1,0}_{p,X}$ for all $n_p^-+1\leq j\leq n_p^-+r$.
This means that \eqref{eq:kernel_hypothesis} is satisfied, so by \eqref{eq:tau_4_hypothesis}, we know that for any $X\in S_0$ we have
  \begin{equation}
  \label{eq:positive_hessian}
    \sum_{j=n_p^-+1}^{n_p^-+r} N^{\bar j j}_p\Hess^b(X,X)\opL(L_j,\bar L_j)=\sum_{j=n_p^-+1}^{n_p^-+r} N^{\bar j j}_p\tau^4(X,X,L_j,\bar L_j)>0.
  \end{equation}

  Let $\mathcal{O}$ denote an open neighborhood of $S_0$ in $S$ on which \eqref{eq:positive_hessian} holds.  On $S\backslash S_0$,
  \[
    \sum_{\{n_p^-+1\leq j,k\leq q:j\leq n_p^-+r\text{ or }k\leq n_p^-+r\}}|\tau^3_p(X,L_j,\bar L_k)|^2>0,
  \]
  so on $S\backslash\mathcal{O}$ this must have a uniform lower bound of $\kappa>0$.  On $\mathcal{O}$, \eqref{eq:upsilon_hessian} and \eqref{eq:positive_hessian} give us \eqref{eq:upsilon_positivity},
  while on $S\backslash\mathcal{O}$, \eqref{eq:upsilon_hessian} gives us
  \begin{multline}
    \Hess^b(X,X)\left(\sum_{j=1}^q\mu_j-\sum_{j,k=1}^{n-1}\Upsilon^{\bar k j}\opL(L_j,\bar L_k)\right)\bigg|_p\geq\\
    2t\eps\kappa+\sum_{j=n_p^-+1}^{n_p^-+r} N^{\bar j j}_p\Hess^b(X,X)\opL(L_j,\bar L_j),
  \end{multline}
  so we may choose
  \[
    t>-\frac{1}{2\eps\kappa}\inf_{X\in S\backslash\mathcal{O}}\sum_{j=n_p^-+1}^{n_p^-+r} N^{\bar j j}_p\Hess^b(X,X)\opL(L_j,\bar L_j)
  \]
  to obtain \eqref{eq:upsilon_positivity} on $S\backslash\mathcal{O}$.

  We have now shown that $\Omega$ satisfies weak $Z(q)$, and the conclusions of the theorem will follow from Theorem 1.1 in \cite{HaRa15}.
\end{proof}

%
%
\section{Indefinite Families of Matrices}
\label{sec:indefinite_subspaces}

Motivated by Theorems \ref{thm:isolated degeneracy} and \ref{thm:isolated degeneracy_improved}, we will develop the theory of indefinite matrices.  First, we will show that we may replace positive definite with positive semi-definite in the definition of indefinite families by restricting to an appropriate subspace.

\begin{lem}
\label{lem:semi_definite_direction}
  Let $\{M_j\}$ be a collection of hermitian $n\times n$ matrices such that $\Span_{\mathbb{R}}\{M_j\}$ is indefinite on some nontrivial subspace $V\subset\mathbb{C}^n$.  If $\sum_j a^j M_j$ is positive semi-definite when restricted to $V$ for some nontrivial collection of real numbers $\{a^j\}$, then $\Span_{\mathbb{R}}\{M_j\}$ is indefinite on the kernel of $\sum_j a^j M_j$ in $V$.
\end{lem}

\begin{proof}
  Let $1\leq\ell\leq n$ denote the dimension of $V$, and let $P_V$ denote an $n\times\ell$ matrix such that the columns form an orthonormal basis for $V$. Our hypotheses are that $\Span_{\mathbb{R}}\{\bar P_V^T M_j P_V\}$ is indefinite on $\C^\ell$
and the $\ell\times\ell$ matrix $\sum_j a^j\bar P_V^T M_j P_V$ is positive semi-definite.
Let $1\leq k\leq \ell-1$ denote the dimension of the kernel of
$\sum_j a^j\bar P_V^T M_j P_V$, and let $P_k$ denote an $\ell\times k$ matrix such that the columns form an orthonormal basis for the kernel of
$\sum_j a^j \bar P_V^T M_j P_V$.
If the $k\times k$ matrix $\sum_j b^j (\bar P_V\bar P_k)^T M_j P_V P_k$ is positive definite for some collection of real numbers $\{b^j\}$, then there exists some $\eps>0$
such that $\sum_j (a^j+\eps b^j)\bar P_V^T M_j P_V$ is also positive definite.  To see this, suppose that $\lambda_1>0$ is the smallest positive eigenvalue of $\sum_j a^j\bar P_V^T M_j P_V$, $\lambda_2>0$ is the smallest eigenvalue of $\sum_j b^j (\bar P_V\bar  P_k)^T M_j P_V P_k$, and $\Lambda$ is the magnitude of the largest eigenvalue of $\sum_j b^j\bar P_V^T M_j P_V$ (in absolute value).  For any nontrivial vector $v\in\mathbb{C}^\ell$, let $v_1=v-P_k\bar P_k^T v$ and $v_2=P_k\bar P_k^T v$.  Then
  \begin{multline*}
    \bar{v}^T\left(\sum_j (a_j+\eps b_j)\bar P_V^T M_j P_V\right)v\geq\\
     \lambda_1\abs{v_1}^2+\eps\lambda_2\abs{v_2}^2+2\eps\re\left(\bar{v}_1^T\left(\sum_j b_j\bar P_V^T M_j P_V\right)v_2\right)-\eps
     \Lambda\abs{v_1}^2\geq\\
     \frac{\lambda_1}{2}\abs{v_1}^2+\eps\lambda_2\abs{v_2}^2-\frac{2\Lambda^2}{\lambda_1}\eps^2\abs{v_2}^2-\eps
     \Lambda\abs{v_1}^2,
  \end{multline*}
  and this will be strictly positive for $0<\eps<\min\set{\frac{\lambda_1}{2\Lambda},\frac{\lambda_1\lambda_2}{2\Lambda^2}}$.  Since $\sum_j (a_j+\eps b_j)\bar P_V^T M_j P_V$ is never positive definite by assumption, we conclude that $\Span_{\mathbb{R}}\{(\bar P_V\bar P_k)^T M_j P_V P_k\}$ is indefinite.
\end{proof}

As a trivial corollary, we obtain our first property of minimal subspaces.
\begin{cor}
\label{cor:minimal}
  Let $\{M_j\}$ be a collection of hermitian $n\times n$ matrices.  If a nontrivial subspace $V\subset\mathbb{C}^n$ is minimal with respect to $\Span_{\mathbb{R}}\{M_j\}$, then no real linear combination of $\{M_j\}$ is positive semi-definite on $V$ unless it is equal to the zero matrix.
\end{cor}

With these tools in place, we are ready to prove our main characterization of indefinite collections of matrices: they are orthogonal to at least one positive semi-definite matrix.  To prove this, it will be helpful to note that if $A=(a_{j\bar k})_{1\leq j,k\leq n}$, then $|A|=\sqrt{\sum_{j,k=1}^n|a_{j\bar k}|^2}$, so
\[
  \sup_{1\leq j,k\leq n}|a_{j\bar k}|\leq|A|\leq n\sup_{1\leq j,k\leq n}|a_{j\bar k}|.
\]
This norm equivalence will be critical in what follows, since it implies that a sequence $\{A_j\}$ converges to $A$ in norm if and only if every element of $\{A_j\}$ converges to the corresponding element of $A$.  In particular, if $A_t$ is a family of hermitian matrices parameterized by $t\in\mathbb{R}$, then $A_t$ is differentiable if and only if every element of $A_t$ is differentiable.

\begin{lem}
\label{lem:orthogonal positive_matrix}
  Let $\{M_j\}$ be a collection of hermitian $n\times n$ matrices and let $V\subset\mathbb{C}^n$ be a nontrivial vector space.  There exists a nontrivial positive semi-definite hermitian $n\times n$ matrix $N$ such that the range of $N$ lies in $V$ and $\left<M_j,N\right>=0$ for all $j$ if and only if $\Span_{\mathbb{R}}\{M_j\}$ is indefinite on $V$.
\end{lem}

\begin{proof}
  We begin with the assumption that $N$ exists.  Let $1\leq\ell\leq n$ be the dimension of $V$ and let $P_V$ be an $n\times\ell$ matrix whose columns form an orthonormal basis for $V$.  We may choose $P_V$ to diagonalize $\bar P_V^T N P_V$.  By assumption, $P_V\bar P_V^T N=N$, so
  \[
    N=P_V\bar P_V^T N P_V\bar P_V^T.
  \]
  Since $N$ is nontrivial and positive semi-definite, at least one eigenvalue of $\bar P_V^T N P_V$ must be positive; suppose it is the first eigenvalue, and denote it $\lambda_1>0$.  Suppose $\sum_j a^j \bar P_V^T M_j P_V$ is positive definite for some collection of real numbers $\{a^j\}$.  Then
  \[
    0=\sum_{j=1}^n a^j\left<M_j,N\right>=\Tr\left(\sum_{j=1}^n a^j(\bar P_V^T M_j P_V)(\bar P_V^T N P_V)\right)\geq\lambda_1\sum_{j=1}^n a^j(\bar P_V^T M_j P_V)_{11}>0,
  \]
  a contradiction.  Hence, $\Span_{\mathbb{R}}\{M_j\}$ is indefinite on $V$.

  Conversely, suppose that $\Span_{\mathbb{R}}\{M_j\}$ is indefinite on $V$.  Let $\tilde{V}\subset V$ be a subspace that is minimal with respect to
$\Span_{\mathbb{R}}\{M_j\}$.  By Corollary \ref{cor:minimal}, no real linear combination of $\{M_j\}$ is positive semi-definite on $\tilde V$ unless it is equal to the zero
matrix.  Let $1\leq\ell\leq n$ denote the dimension of $\tilde V$, and let $P_{\tilde V}$ denote an $n\times\ell$ matrix with columns forming an orthonormal basis
for $\tilde V$.  If there exists a nontrivial positive semi-definite $\ell\times\ell$ matrix $\tilde N$ such that
$\left<\bar P_{\tilde V}^T M_j P_{\tilde V},\tilde N\right>=0$ for every $j$, then $N=P_{\tilde V}\tilde N\bar P_{\tilde V}^T$ will satisfy the conclusion of the lemma.  Hence, it will suffice to prove the result after restricting to $\tilde V$.

  For ease of notation, we assume that we have already restricted to the subspace described in the previous paragraph, and that no linear combination of $\{M_j\}$ is positive semi-definite unless it is equal to the zero matrix.  Without loss of generality, we may assume that $\{M_j\}$ is linearly independent, so that no nontrivial real linear combination of $\{M_j\}$ is positive semi-definite.  Let $S$ denote the set of all unit-length positive semi-definite hermitian $n\times n$ matrices, and note that $\left<N_1,N_2\right>\geq 0$ for any $N_1,N_2\in S$, so $0\leq\arccos\left<N_1,N_2\right>\leq\frac{\pi}{2}$.  Since $S$ is a compact set that does not intersect $\Span_{\mathbb{R}}\{M_j\}$ by assumption, we can set $N$ equal to an element of $S$ maximizing the distance from $S$ to
$\Span_{\mathbb{R}}\{M_j\}$.  Let $\tilde{M}$ denote the unique element of $\Span_{\mathbb{R}}\{M_j\}$ satisfying
$\dist(\tilde{M},N)=\dist(\Span_{\mathbb{R}}\{M_j\},N)$.  Our goal is to show that $\tilde{M}=0$.  Note that since $\tilde{M}$ must represent the orthogonal projection of $N$ onto $\Span_{\mathbb{R}}\{M_j\}$, we have $\left<\tilde{M}-N,M\right>=0$ for all $M\in\Span_{\mathbb{R}}\{M_j\}$.  On the other hand, suppose that $N'\in S$ is linearly independent from $N$, and set $N_t=\cos t N+\sin t\frac{N'-\left<N',N\right>N}{\sqrt{1-\left<N,N'\right>^2}}$.  For $0\leq t\leq \arccos\left<N,N'\right>$, we have $N_t\in S$ as well.  Let $\tilde{M}'\in\Span_{\mathbb{R}}\{M_j\}$ minimize the distance to $N'$, so that $\left<\tilde{M}'-N',M\right>=0$ for all $M\in\Span_{\mathbb{R}}\{M_j\}$.  Since $\tilde{M}_t=\cos t \tilde{M}+\sin t\frac{\tilde{M}'-\left<N',N\right>\tilde{M}}{\sqrt{1-\left<N,N'\right>^2}}$ is an element of $\Span_{\mathbb{R}}\{M_j\}$ satisfying $\left<\tilde{M}_t-N_t,M\right>=0$ for all $M\in\Span_{\mathbb{R}}\{M_j\}$, it must minimize the distance to $N_t$.  By assumption, $\abs{\tilde{M}-N}\geq\abs{\tilde{M}_t-N_t}$ for all $0\leq t\leq 1$.  Using $\left<\tilde{M}-N,\tilde{M}'\right>=0$, we compute
  \[
    \abs{\tilde{M}_t-N_t}^2=\cos^2 t\abs{\tilde{M}-N}^2-2\sin t\cos t\frac{\left<\tilde{M}-N,N'\right>+\left<N',N\right>\abs{\tilde{M}-N}^2}{\sqrt{1-\left<N,N'\right>^2}}+O(\sin^2 t),
  \]
  so
  \[
    \abs{\tilde{M}-N}^2-\abs{\tilde{M}_t-N_t}^2=2\sin t\cos t\frac{\left<\tilde{M}-N,N'\right>+\left<N',N\right>\abs{\tilde{M}-N}^2}{\sqrt{1-\left<N,N'\right>^2}}+O(\sin^2 t).
  \]
  Since this quantity is nonnegative, we may divide by $0<t\leq \arccos\left<N,N'\right>$ and let $t\rightarrow 0^+$ to conclude $\left<\tilde{M}-N,N'\right>+\left<N',N\right>\abs{\tilde{M}-N}^2\geq 0$.  Since $\left<\tilde{M}-N,\tilde{M}\right>=0$ implies $\abs{\tilde{M}-N}^2=-\left<\tilde{M}-N,N\right>$, we can simplify our conclusion to obtain
  \begin{equation}
    \label{eq:N_maximizer}
    \left<\tilde{M}-\left<\tilde{M},N\right>N,N'\right>\geq 0
  \end{equation}
  for any $N'\in S$ (observe that this is trivially true if $N'$ and $N$ are linearly dependent).

  Let $v$ be any unit length eigenvector of $N$ with eigenvalue $\lambda>0$ and let $u$ be any vector.  Then
  \[
    N+tu\bar{v}^T+tv\bar{u}^T+t^2\lambda^{-1}u\bar{u}^T=N-\lambda v\bar{v}^T+\lambda(v+\lambda^{-1}tu)(\bar{v}^T+\lambda^{-1}t\bar{u}^T)
  \]
  is positive semi-definite for any $t\in\mathbb{R}$ with norm
  \[
    \abs{N+tu\bar{v}^T+tv\bar{u}^T+t^2\lambda^{-1}u\bar{u}^T}
    =\sqrt{1+4t\lambda\re(\bar{u}^T v)+O(t^2)}
  \]
  so we can obtain an element of $S$ via
  \begin{multline*}
    A_t=\frac{N+tu\bar{v}^T+tv\bar{u}^T+t^2\lambda^{-1}u\bar{u}^T}{\abs{N+tu\bar{v}^T+tv\bar{u}^T+t^2\lambda^{-1}u\bar{u}^T}}\\
    =\frac{N}{\abs{N+tu\bar{v}^T+tv\bar{u}^T+t^2\lambda^{-1}u\bar{u}^T}}+(1-2t\lambda\re(\bar{u}^T v))t(u\bar{v}^T+v\bar{u}^T)
    +t^2\lambda^{-1}u\bar{u}^T+O(t^3).
  \end{multline*}
  Since \eqref{eq:N_maximizer} only depends on the component of $N'$ that is orthogonal to $N$, we can substitute $A_t$ in place of $N'$ to obtain
  \[
    \left<\tilde{M}-\left<\tilde{M},N\right>N,(1-2t\lambda\re(\bar{u}^T v))t(u\bar{v}^T+v\bar{u}^T)+t^2\lambda^{-1}u\bar{u}^T\right>\geq -O(t^3)
  \]
  Since this must hold for all $t>0$ and $t<0$, the terms of order $t$ must vanish, and hence $\left<\tilde{M}-\left<\tilde{M},N\right>N,(u\bar{v}^T+v\bar{u}^T)\right>=0$.  We are left with
  \[
    t^2\lambda^{-1}\bar{u}^T\left(\tilde{M}-\left<\tilde{M},N\right>N\right)u\geq -O(t^3).
  \]
  Since this holds for all $t\in\mathbb{R}$ and all vectors $u$, $\tilde{M}-\left<\tilde{M},N\right>N$ must be positive semi-definite.  Since $\left<\tilde{M},N\right>=|\tilde{M}|^2$, $\tilde{M}$ must be positive semi-definite.  By assumption, this is only possible if $\tilde{M}=0$.  Since the origin is the closest point in $\Span_{\mathbb{R}}\{M_j\}$ to $N$, $N$ must be orthogonal to $\Span_{\mathbb{R}}\{M_j\}$, and our conclusion follows.

\end{proof}

In general, the matrix $N$ obtained in Lemma \ref{lem:orthogonal positive_matrix} may not be unique.  As a first step to understanding what it means for $N$ to be unique, we consider one consequence of nonuniqueness:

\begin{lem}
\label{lem:nonunique_minimal_spaces}
  Let $\{M_j\}$ be a collection of hermitian $n\times n$ matrices.  Suppose there exist two linearly independent positive semi-definite matrices $N_1$ and $N_2$ with the property that $\left<M_j,N_1\right>=\left<M_j,N_2\right>=0$ for all $j$.  Then there exist two nontrivial subspaces $V_1,V_2\subset\Span_{\mathbb{C}}\{\Ran(N_1),\Ran(N_2)\}$ such that $V_1\cap V_2=\emptyset$ and each subspace is minimal with respect to $\Span_{\mathbb{R}}\{M_j\}$.
\end{lem}

\begin{proof}
Let $\tilde V_1=\Ran(N_1)$ and $\tilde V_2=\Ran(N_2)$.  By Lemma \ref{lem:orthogonal positive_matrix}, $\Span_{\mathbb{R}}\{M_j\}$ is indefinite on each of these subspaces.
Consequently, there exist subspaces $V_j'\subset \tilde V_j$ for $j\in\{1,2\}$ that are minimal with respect to $\Span_{\mathbb{R}}\{M_j\}$.
Suppose that there exists a unique minimal subspace $V_3\subset\Span_{\mathbb{C}}\{\Ran(N_1),\Ran(N_2)\}$ with respect to $\Span_{\mathbb{R}}\{M_j\}$.  By uniqueness,
$V_3=V_1'=V_2'$, which means
$V_3\subset\tilde V_1\cap\tilde V_2$.  By Lemma \ref{lem:orthogonal positive_matrix}, there must exist a nontrivial positive semi-definite hermitian matrix $N_3$ such that
$\Ran(N_3)\subset V_3$ and $\left<M_j,N_3\right>=0$ for all $j$.
Since $N_1$ and $N_2$ are linearly independent, $N_3$ must be linearly independent from either $N_1$ or $N_2$.  Without loss of generality, suppose $N_3$ is linearly independent from $N_1$.  Define
  \[
    \tilde\lambda=\inf\set{\lambda\in\mathbb{R}:\lambda N_1-N_3\text{ is positive semi-definite}}.
  \]
Since $\Ran(N_3)\subset\Ran(N_1)$, $\lambda N_1-N_3$ is positive semi-definite for all $\lambda>0$ sufficiently large, so $\tilde\lambda$ is a finite, positive number.
Since $N_1$ and $N_3$ are linearly independent, $\tilde\lambda N_1-N_3$ is nontrivial.  If $\Ran(N_3)\subset\Ran(\tilde\lambda N_1-N_3)$, then we could further decrease $\tilde\lambda$, contradicting its definition, so we must have $\Ran(N_3)\not\subset\Ran(\tilde\lambda N_1-N_3)$.  Hence, Lemma \ref{lem:orthogonal positive_matrix} tells us that $\Span_{\mathbb{R}}\{M_j\}$ is indefinite on $\Ran(\tilde\lambda N_1-N_3)$, but $V_3$ is not a subset of $\Ran(\tilde\lambda N_1-N_3)$, so there must exist a subspace of $\Ran(\tilde\lambda N_1-N_3)$ that is minimal with respect to $\Span_{\mathbb{R}}\{M_j\}$, contradicting our assumption.  Hence, there must exist at least two subspaces that are minimal with respect to $\Span_{\mathbb{R}}\{M_j\}$.
\end{proof}

We are now ready to completely characterize minimal subspaces.
\begin{lem}
\label{lem:minimal_with_N}
  Let $\{M_j\}$ be a collection of hermitian $n\times n$ matrices.  A nontrivial subspace $V\subset\mathbb{C}^n$ is minimal with respect to $\Span_{\mathbb{R}}\{M_j\}$ if and only if there exists a unique (up to a positive scalar multiple) positive semi-definite hermitian $n\times n$ matrix $N$ such that the range of $N$ equals $V$ and $\left<M_j,N\right>=0$ for all $j$.
\end{lem}

\begin{proof}
 Suppose $V$ is minimal with respect to $\Span_{\mathbb{R}}\{M_j\}$.  Lemma \ref{lem:orthogonal positive_matrix} guarantees the existence of a positive semi-definite hermitian $n\times n$ matrix $N$ such that the range of $N$ is contained in $V$ and $\left<M_j,N\right>=0$ for all $j$.  If the range of $N$ is a proper subspace of $V$, then Lemma \ref{lem:orthogonal positive_matrix} also implies that $\Span_{\mathbb{R}}\{M_j\}$ is indefinite on the range of $N$, contradicting the definition of minimality.  Hence, $V$ is equal to the range of $N$.  By Lemma \ref{lem:nonunique_minimal_spaces}, $N$ must be unique up to a positive scalar multiple.

  Suppose that $V$ is not minimal with respect to $\Span_{\mathbb{R}}\{M_j\}$, but there exists a unique (up to a positive scalar multiple) positive semi-definite hermitian $n\times n$ matrix $N$ such that the range of $N$ equals $V$ and $\left<M_j,N\right>=0$ for all $j$.  By Lemma \ref{lem:orthogonal positive_matrix}, $\Span_{\mathbb{R}}\{M_j\}$ is indefinite on $V$.  Since $V$ is not minimal with respect to $\Span_{\mathbb{R}}\{M_j\}$ there exists a nontrivial proper subspace $\tilde{V}\subset V$ such that $\Span_{\mathbb{R}}\{M_j\}$ is indefinite on $\tilde{V}$.  By Lemma \ref{lem:orthogonal positive_matrix}, there exists a nontrivial positive semi-definite hermitian $n\times n$ matrix $\tilde{N}$ such that the range of $\tilde{N}$ is contained in $\tilde{V}$ and $\left<M_j,\tilde{N}\right>=0$ for all $j$.  Since the range of $\tilde{N}$ is a proper subset of the range of $N$, $\tilde{N}$ and $N$ must be linearly independent.  For any $t\geq0$, $N+t\tilde{N}$ is a matrix linearly independent from $N$ with range equal to $V$ such that $\left<M_j,N+t\tilde{N}\right>=0$ for all $j$, contradicting the uniqueness of $N$.
\end{proof}

\subsection{Examples.}

To illustrate the previous results, we consider collections of $2\times 2$ matrices that are indefinite on $\mathbb{C}^2$.  The only nontrivial proper subspaces of $\mathbb{C}^2$ are complex lines through the origin, and any nontrivial matrix on a line is either positive definite or negative definite, and a negative definite matrix is a scalar multiple of a positive definite matrix.  Hence, our collection can be indefinite on a complex line if and only if it is trivial on that complex line.

Consider $\{M_j\}=\left\{\begin{pmatrix}-1&0\\0&1\end{pmatrix},\begin{pmatrix}0&1\\1&0\end{pmatrix},\begin{pmatrix}0&i\\-i&0\end{pmatrix}\right\}$.  This collection is not indefinite on any nontrivial proper subspace of $\mathbb{C}^2$, and the only positive semi-definite matrices orthogonal to the span of these matrices are positive scalar multiples of the identity matrix.

Suppose, instead, that we consider $\left\{\begin{pmatrix}0&1\\1&0\end{pmatrix},\begin{pmatrix}0&i\\-i&0\end{pmatrix}\right\}$.  This collection is indefinite on two proper subspaces: the span of $\begin{pmatrix}1\\0\end{pmatrix}$ and the span of $\begin{pmatrix}0\\1\end{pmatrix}$.  The positive semi-definite matrices orthogonal to the span of these matrices are of the form $N_{a,b}=a\begin{pmatrix}1&0\\0&0\end{pmatrix}+b\begin{pmatrix}0&0\\0&1\end{pmatrix}$, where $a$ and $b$ are both nonnegative and at least one is positive.  Note that the range of $N_{1,0}$ equals the span of  $\begin{pmatrix}1\\0\end{pmatrix}$, and the range of $N_{0,1}$ equals the span of $\begin{pmatrix}0\\1\end{pmatrix}$.

If we further restrict to $\left\{\begin{pmatrix}0&i\\-i&0\end{pmatrix}\right\}$, then this collection is indefinite on infinitely many proper subspaces of $\mathbb{C}^2$: each is given by the span of $\begin{pmatrix}\cos\theta\\\sin\theta\end{pmatrix}$ where $\theta\in\mathbb{R}$.  The positive semi-definite matrices orthogonal to the span of these matrices are of the form $N_{a,b,c}=a\begin{pmatrix}1&0\\0&0\end{pmatrix}+b\begin{pmatrix}0&0\\0&1\end{pmatrix}+c\begin{pmatrix}0&1\\1&0\end{pmatrix}$, where at least one of $a$ and $b$ is positive and $ab\geq c^2$.  Note that the range of $N_{\cos^2\theta,\sin^2\theta,\sin\theta\cos\theta}$ is equal to the span of $\begin{pmatrix}\cos\theta\\\sin\theta\end{pmatrix}$.

Finally, consider $\left\{\begin{pmatrix}1&0\\0&0\end{pmatrix}\right\}$.  The only proper subspace on which this collection is indefinite is the span of $\begin{pmatrix}0\\1\end{pmatrix}$.  The only positive semi-definite matrices that are orthogonal to this matrix are positive scalar multiples of $\left\{\begin{pmatrix}0&0\\0&1\end{pmatrix}\right\}$.  Once again, the range of $N$ is equal to the span of $\begin{pmatrix}0\\1\end{pmatrix}$.

\subsection{Proofs of Theorems \ref{thm:isolated degeneracy} and \ref{thm:isolated degeneracy_improved}}
\label{sec:more_proof}

\begin{proof}[Proof of Theorem \ref{thm:isolated degeneracy}]
Let $n_p^-$ and $\{L_1^p,\ldots,L_{n-1}^p\}$ be as in the statement of Theorem \ref{thm:main_theorem}.  From Corollary \ref{cor:counting_negative_eigenvalues}, we know that $\Span_{X\in T_p(\partial\Omega)}\tau^3_p(X,\cdot,\cdot)$
is indefinite on $K^{1,0}_p$ at every point $p\in\bd\Om$ at which $\opL$ has exactly $n-q-1$ positive eigenvalues.  At such a $p$, it follows from
Lemma \ref{lem:orthogonal positive_matrix} that there exists a $(q-n_p^-)\times(q-n_p^-)$ positive semi-definite matrix $N_p$ satisfying \eqref{eq:tau_3_hypothesis}.

Suppose \eqref{eq:kernel_hypothesis} is satisfied for some $X\in T_p(\partial\Omega)$.  If we identify the range of $N_p$ with the span of the vector fields in \eqref{eq:kernel_hypothesis}, then $K^{1,0}_{p,X}$ is in the range of $N_p$, so
$\Span_{Y\in T_p(\partial\Omega)}\tau^3_p(Y,\cdot,\cdot)$ must be indefinite on $K^{1,0}_{p,X}$ by \eqref{eq:tau_3_hypothesis} and Lemma \ref{lem:orthogonal positive_matrix}.
By hypothesis, this means that there exists $Y\in T_p(\partial\Omega)$ such that
$\tau^3_p(Y,\cdot,\cdot)+\tau^4_p(X,X,\cdot,\cdot)$ is positive definite on $K^{1,0}_{p,X}$.  Hence, using Remark \ref{rem:kernel_hypothesis}, we have
\[
  0<\sum_{j,k=n_p^-+1}^{q}N_p^{\bar k j}(\tau^3_p(Y,L_j^p,\bar L_k^p)+\tau^4_p(X,X,L_j^p,\bar L_k^p)),
\]
so \eqref{eq:tau_3_hypothesis} implies that \eqref{eq:tau_4_hypothesis} is satisfied.
\end{proof}

\begin{proof}[Proof of Theorem \ref{thm:isolated degeneracy_improved}]

  We follow the proof of Theorem \ref{thm:isolated degeneracy}, with the exception that Lemma \ref{lem:minimal_with_N} and our hypotheses guarantees that we may choose $N_p$ so that $\Ran(N_p)=V_p$ (under the usual identification suggested by \eqref{eq:kernel_hypothesis}).  For $X\in T_p(\partial\Omega)$ satisfying \eqref{eq:kernel_hypothesis}, since $\Ran(N_p)\subset K^{1,0}_{p,X}$ and $V_p=\Ran(N_p)$, our hypotheses guarantees that there exists $Y\in T_p(\partial\Omega)$ such that $\tau^3_p(Y,\cdot,\cdot)+\tau^4_p(X,X,\cdot,\cdot)$ is positive definite on $V_p$, and \eqref{eq:tau_3_hypothesis} follows by the same reasoning as before.  The remainder of the proof is identical.
\end{proof}

%
%
\section{Examples}\label{sec:example}

We will provide several examples of weak $Z(2)$ domains in $\mathbb{C}^3$.  Our examples will all contain isolated degeneracies at the origin. Since the degeneracies are isolated, these local examples can easily be realized as bounded $Z(2)$ domains in a simple special case of Proposition 6.6 in \cite{HaRa15}.  Let $z_j=x_j+iy_j$.  In a neighborhood of the origin, we will write
\[
  \rho(z)=-y_3+P(z_1,z_2,x_3),
\]
where $P$ is a real quartic polynomial such that $P(0)=0$, $\nabla P(0)=0$, and $\Hess(X,X)P(0)=0$ for all $X\in T_0(\mathbb{C}^n)$.
Using \eqref{eq:boundary_hessian}, $\Hess^b(X,X)=\Hess(X,X)+O(|z|)$, so we will be able to neglect the difference between these operators.
Since the Levi form has no non-trivial eigenvalues at the origin, $T^{1,0}_{0}(\partial\Omega)=\tilde K^{1,0}_{0}(\partial\Omega)$,
so we can neglect this space.
We can choose our orthonormal coordinates to satisfy $L_1=\frac{\partial}{\partial z_1}+O(|z|^2)$ and $L_2=\frac{\partial}{\partial z_2}+O(|z|^2)$.
In these coordinates, we can easily compute $\tau^3_0(X,L_j,\bar L_k)=X\frac{\partial^2 P}{\partial z_j\partial\bar z_k}\big|_0$ and,
when it is defined, $\tau^4_0(X,X,L_j,\bar L_k)=\Hess(X,X)\frac{\partial^2 P}{\partial z_j\partial\bar z_k}\big|_0$.

\begin{ex}\label{ex:simple example}
For our first example, let
\[
  P(z)=\re(z_1^2\bar z_2)-x_3|z_1|^2+x_3|z_2|^2-|z_1|^2|z_2|^2+|z_2|^4,
\]
so the matrix representing the Levi form (still denoted $\opL$ by an abuse of notation) is
\[
 \opL=\begin{pmatrix}-x_3-|z_2|^2&z_1-\bar z_1 z_2\\\bar z_1-z_1\bar z_2&x_3-|z_1|^2+4|z_2|^2\end{pmatrix}+O(|z|^3)
\]
Then the nontrivial values of $\tau^3_p$ can be computed from
\begin{align*}
  \tau^3_0\left(\frac{\partial}{\partial x_1},\cdot,\cdot\right)&=\begin{pmatrix}0&1\\1&0\end{pmatrix},\\
  \tau^3_0\left(\frac{\partial}{\partial y_1},\cdot,\cdot\right)&=\begin{pmatrix}0&i\\-i&0\end{pmatrix},\\
  \tau^3_0\left(\frac{\partial}{\partial x_3},\cdot,\cdot\right)&=\begin{pmatrix}-1&0\\0&1\end{pmatrix}.
\end{align*}
We note that matrices of the form $\begin{pmatrix}-s&w\\\bar w&s\end{pmatrix}$ where $w\in\mathbb{C}$ and $s\in\mathbb{R}$ are nondegenerate unless $w=0$ and $s=0$, so $K^{1,0}_{p,X}$ is trivial unless $X|_0=a\frac{\partial}{\partial x_2}+b\frac{\partial}{\partial y_2}$, in which case $K^{1,0}_{p,X}=\Span\{L_1,L_2\}$.  For such $X$, we have
\[
  \tau^4_0(X,X,\cdot,\cdot)=\begin{pmatrix}-2(a^2+b^2)&0\\0&8(a^2+b^2)\end{pmatrix}.
\]
This is never positive definite, but
\[
  \tau^4_0(X,X,\cdot,\cdot)-4(a^2+b^2)\tau^3_0\left(\frac{\partial}{\partial x_3},\cdot,\cdot\right)=\begin{pmatrix}2(a^2+b^2)&0\\0&4(a^2+b^2)\end{pmatrix},
\]
which is positive definite when $x\neq 0$, so by Theorem \ref{thm:isolated degeneracy} we must have weak $Z(2)$ in a neighborhood of the origin.
Not coincidentally, for any $w\in\mathbb{C}$, $\opL$ is positive definite for $t$ sufficiently small along the path $\gamma_t=(0,tw,-2t^2|w|^2+iP(0,tw,-2t^2|w|^2))$.  Following the proof of Theorem \ref{thm:isolated degeneracy}, we can see that if $0<\eps<1$
\[
  \Upsilon=\begin{pmatrix}1-\eps(1-t x_3)&-\eps t z_1\\-\eps t\bar z_1&1-\eps(1+t x_3)\end{pmatrix}
\]
is positive definite on a neighborhood of the origin, $I-\Upsilon$ is also positive definite on a neighborhood of the origin, $\Tr\Upsilon<2$, and $\Upsilon$ satisfies
\[
  \Tr(\opL)-\Tr(\Upsilon\opL)=\eps((2t-1) |z_1|^2+3|z_2|^2+2t x_3^2)+O(|z|^3),
\]
so this quantity will be positive in a neighborhood of the origin provided that $t>\frac{1}{2}$.
\end{ex}

\begin{ex}\label{ex:complicated example}
For our second example, let
\[
  P(z)=\re(z_1^2\bar z_2)-|z_1|^2|z_2|^2+\frac{1}{4}|z_2|^4-|z_1|^2 x_3^2+|z_2|^2 x_3^2,
\]
so
\[
\opL=\begin{pmatrix}-|z_2|^2-x_3^2&z_1-\bar z_1 z_2\\\bar z_1-z_1\bar z_2&-|z_1|^2+|z_2|^2+x_3^2\end{pmatrix}+O(|z|^3).
\]
This time, the only nontrivial value of $\tau^3_p$ are
\begin{align*}
  \tau^3_0\left(\frac{\partial}{\partial x_1},\cdot,\cdot\right)&=\begin{pmatrix}0&1\\1&0\end{pmatrix},\\
  \tau^3_0\left(\frac{\partial}{\partial y_1},\cdot,\cdot\right)&=\begin{pmatrix}0&i\\-i&0\end{pmatrix},\\
\end{align*}
so $K^{1,0}_{p,X}$ is trivial unless $X|_0=a\frac{\partial}{\partial x_2}+b\frac{\partial}{\partial y_2}+c\frac{\partial}{\partial x_3}$, in which case $K^{1,0}_{p,X}=\Span\{L_1,L_2\}$.  For such $X$, we have
\[
  \tau^4_0(X,X,\cdot,\cdot)=\begin{pmatrix}-2(a^2+b^2+c^2)&0\\0&2(a^2+b^2+c^2)\end{pmatrix}.
\]
This is never positive definite, and adding a component of $\tau^3_0$ will only decrease the size of the determinant.  However, $\Span_{Y\in T_0(\partial\Omega)}\tau^3_0(Y,\cdot,\cdot)$ is indefinite on $V_0=\Span\{L_2\}$, so $V_0$ is a minimal subspace with respect to $\Span_{Y\in T_0(\partial\Omega)}\tau^3_0(Y,\cdot,\cdot)$.  Whenever $V_0\subset K^{1,0}_{p,X}$, the restriction of $\tau^4_0(X,X,\cdot,\cdot)$ to $V_0$ is simply $\left(2(a^2+b^2+c^2)\right)$, which is positive definite.  Hence, even though Theorem \ref{thm:isolated degeneracy} fails,
Theorem \ref{thm:isolated degeneracy_improved} will still give us weak $Z(2)$ in a neighborhood of the origin.  In fact, we can see that if $0<\eps<1$
\[
  \Upsilon=\begin{pmatrix}1-2\eps t^2|z_1|^2&-\eps t z_1\\-\eps t\bar z_1&1-\eps\end{pmatrix}
\]
is positive definite on a neighborhood of the origin, $I-\Upsilon$ is also positive definite on a neighborhood of the origin, $\Tr\Upsilon<2$, and $\Upsilon$ satisfies
\[
  \Tr(\opL)-\Tr(\Upsilon\opL)=\eps((2t-1)|z_1|^2+|z_2|^2+x_3^2)+O(|z|^3),
\]
so this quantity will be positive in a neighborhood of the origin provided that $t>\frac{1}{2}$.
\end{ex}

\bibliographystyle{alpha}
\bibliography{mybib_5_10_18}
\end{document}